\documentclass[11pt]{amsart}
\usepackage{graphicx}
\usepackage{amssymb,amsmath,amsthm}
\usepackage{hyperref}
\usepackage{mathtools}
\usepackage{bm}
\usepackage{color}
\newtheorem{theorem}{Theorem}[section]

\newtheorem{lemma}[theorem]{Lemma}

\newtheorem{corollary}[theorem]{Corollary}
\theoremstyle{definition}

\numberwithin{equation}{section}
\numberwithin{table}{section}
\numberwithin{figure}{section}
\everymath{\displaystyle}
\title
[A posteriori error estimates by preconditioning]{A posteriori error estimates of finite element methods by preconditioning}\thanks{The work of Zikatanov was supported in part by NSF grants DMS-1720114 and DMS-1819157}
\author{Yuwen Li}
\address{Pennsylvania State University}
\email{yuwli@psu.edu}
\author{Ludmil Zikatanov}
\address{Pennsylvania State University}
\email{ludmil@psu.edu}
\begin{document}

\begin{abstract}
We present a framework that relates preconditioning with a posteriori error estimates in finite element methods. In particular, we use standard tools in subspace correction methods to obtain reliable and efficient error estimators. As a simple example, we recover the classical residual error estimators for the second order elliptic equations.
\end{abstract}

\maketitle
\section{Introduction}
Adaptive finite element methods (AFEMs) have been an active research area since the pioneering work \cite{BR1978}. In contrast to finite elements based on quasi-uniform meshes, AFEMs produce a sequence of locally refined grids that is able to resolve the singularity arising from irregular data in the underlying boundary value problems. Readers are referred to e.g., \cite{BS2001,NSV2009,Verfurth2013} for a thorough introduction. Among the key concepts in AFEMs, a posteriori error estimates are the building block for comparing errors on different elements and marking elements with large errors for refinement.
For details on various AFEM error estimation techniques we refer to works on: explicit residual estimators \cite{Verfurth2013}; implicit estimators based on local problems \cite{BR1978,BankWeiser,CF1999,MNS2003}; recovery-based estimators; \cite{ZZ1992b,CB2002,BX2003a,BX2003b,ZhangNaga2005}; hierarchical basis estimators \cite{bank-acta,BankSmith,LiOvall,HNJ2017}; functional estimators \cite{Repin2008}; and equilibrated estimators  \cite{AO2000,LW2004,Schoberl2008b,Ern2015}.

On the other hand, parallel with the development of AFEMs, there are also substantial research efforts in studying efficient preconditioning, which is a technique for approximating the inverse of a differential operator. Usually, such approximations are aimed at accelerating Krylov subspace iterative methods for solving  linear systems resulting from discretized partial differential equations. Popular techniques used for preconditioning include e.g., multigrid~\cite{MG-book,achi,hackbusch-first,hackbusch-book,yserentant-acta,xz-AMG} and domain decomposition/subspace correction methods~\cite{dryja-widlund,widlund-book,xu-siam-review}. In practice, subspace correction methods provide an efficient way of reducing the condition number of a large-scale but finite-dimensional linear system. However, the analysis of uniform convergence rate of those methods often benefits from the general setting of infinite-dimensional Hilbert spaces (see, for example,~\cite{sergey-fictitious-space,griebel-oswald,xz-jams}).

In this paper we present a general framework relating abstract operator preconditioning~\cite{xu-siam-review,griebel-oswald,xz-jams,LoghinWathen2004,MardalWinther2011} to a posteriori error estimates. In particular, we shall show that such standard techniques for developing preconditioners also yield reliable and efficient error estimators. Here, for clarity of presentation, we focus on the symmetric and positive-definite problems although extensions to more general cases are definitely within reach. As a simple example, with this framework, we are able to recover the classical residual error estimators for elliptic equations in primal form.

The rest of this paper is organized as follows. In section \ref{secpre}, we set up the model variational problem and
define the operator notation which is convenient when constructing preconditioners. In section \ref{secmain}, we develop the main theory on posteriori error estimates via preconditioning. Section \ref{secex} is devoted to the example of second order elliptic equation that illustrates the aforementioned  abstract theory. Concluding remarks are found in Section \ref{secconclusion}.

\section{Preliminaries}\label{secpre}
Let $V$ be a Hilbert space and $V^\prime$ denote the dual space of $V$. Let $a: V\times V\rightarrow\mathbb{R}$ be a continuous bilinear form and $f\in V'$.
We consider the following variational problem: Find $u\in V$ such that for all $v\in V$
\begin{equation}\label{variational}
a(u,v)=\langle f,v\rangle.
\end{equation}
Here $\langle\cdot,\cdot\rangle$ is the duality pairing between $V^{\prime}$ and $V$. Let $\|\cdot\|_V$  denote the norm on $V$ and $\|\cdot\|_{V^\prime}$ the dual norm of $V^\prime$. For simplicity, we assume that the bilinear form $a(\cdot,\cdot)$ is symmetric and positive-definite (SPD). The continuity and positive-definiteness of $a(\cdot,\cdot)$ imply
\begin{subequations}
\begin{align}
&a(v,w)\leq\overline{\alpha}\|v\|_V\|w\|_V,\label{bdd}\\
&a(v,v)\geq\underline{\alpha}\|v\|_V^2,\label{coercive}
\end{align}
\end{subequations}
for all $v, w\in V,$
where $\overline{\alpha}, \underline{\alpha}>0$ are absolute constants.
Such a bilinear form naturally defines a bounded isomophism $A: V\rightarrow V'$ for which we have
\begin{equation*}
\langle A v, w  \rangle:= a(v,w),\quad\forall v,w\in V.
\end{equation*}
Hence,
\eqref{variational} is equivalent to the operator equation
\begin{equation}
    Au=f.
\end{equation}
\eqref{bdd} and \eqref{coercive} imply that $A$ induces the inner product $\langle A\cdot,\cdot\rangle$ on $V$. For all $v\in V,$ the $A$-norm on $V$ is defined as $\|v\|_A:=\langle Av,v\rangle^\frac{1}{2}$, which is equivalent to the $V$-norm.

\subsection{Approximation from a subspace}
Let us consider a general case where we approximate the solution to~\eqref{variational} by restricting it to a subspace $V_h\subset V$, namely:
Find $u_h\in V_h$ such that
\begin{equation}\label{disvar}
a(u_h,v)=\langle f,v\rangle\text{ for all }v\in V_h.
\end{equation}
Note that the subspace  $V_h$ does not even have to be finite dimensional, although it usually is in applications.
It follows from \eqref{bdd}, \eqref{coercive} and the well-known Lax--Milgram theorem that \eqref{disvar} admits a unique solution.

For such a subspace $V_h\subset V$, we consider the natural inclusion $I_h: V_h\hookrightarrow V$ and its adjoint
$Q_h:=I_h': V'\to V_h'$ defined as
\begin{equation*}
\langle Q_hg, v_h \rangle= \langle g, I_hv_h \rangle\quad \text{ for all }g\in V^\prime\text{ and }v_h\in V_h.
\end{equation*}

We introduce the operator $A_h:= Q_hAI_h: V_h\rightarrow V_h'$ which approximates $A$ on $V_h$. In this way, the discrete problem \eqref{disvar} reads
\begin{equation*}
A_hu_h=Q_hf.
\end{equation*}

\section{A posteriori error estimates by preconditioning}\label{secmain}
A posteriori error estimates are of the form
\begin{align*}
    C_1\eta_h\leq\|u-u_h\|_V\leq C_2\eta_h,
\end{align*}
where $C_1, C_2$ are absolute positive constants and $\eta_h$ is computed from $u_h$. In AFEMs, $\eta_h$ is the sum of error indicators on all elements. The local error indicators can be used to compare errors on different elements and those elements with large errors will be refined. In this way, the errors estimated by $\eta_h$ are equidistributed over all elements in the mesh. The optimal computational complexity of AFEMs is often attributed to the aforementioned equidistribution of errors. Rigorous analysis of convergence and optimality of AFEMs can be found in e.g., \cite{Dorfler1996,MNS2000,BDD2004,Stevenson2007,CKNS2008}.

\subsection{Links with operator preconditioning}
Let
\begin{align*}
    e&:=u-u_h,\\
    r&:={f-Au_h}\in V^\prime.
\end{align*}
Clearly, from our discusion above, it follows that constructing {a posteriori} error estimators is equivalent to estimating a norm of the error $e=A^{-1}r$ by computable bounds. We note, however, that a direct computation of the norm of $A^{-1}r$ will be, in general, impossible or too expensive, 
since one needs to compute the action of $A^{-1}$ on $r$. As we pointed out in the introduction,
approximating such action has been also studied for several decades and is known as as \emph{preconditioning}.
Following this simple observation we now borrow some simple ideas from this field and apply them in constructing a posteriori error estimators.

First, we need a bounded isomorphism (the \emph{preconditioner}) $B: V'\to V$,  whose particular form will be given later. For the time being we only assume that $B$ is bounded and SPD, i.e., $\langle\cdot,B\cdot\rangle$ is an inner product on $V^\prime$. Let $S: V'\to V$ be a SPD operator, which we will refer to as ``the \emph{smoother}" and is such that its range approximates well the high frequency part of the range of $A^{-1}$, i.e., the result of the action $Sr$ provides a good approximation to the high frequency components of the error. Now, a simple choice for $B$ is
\begin{equation*}
B:=S + I_h A_h^{-1}Q_h,
\end{equation*}
which is known as \emph{additive Schwarz preconditioner}. Just to simplify the presentation, we will not consider the multiplicative preconditioner in this paper although following the abstract framework  developed in~\cite{griebel-oswald,xz-jams} similar results can also be obtained in the multiplicative case as well. Let $\underline{\beta}, \overline{\beta}$ be two positive absolute constants. We say that $B$ is a preconditioner for $A$ provided there exist constants
$\underline{\beta}>0$ and $\overline{\beta}<\infty$, such that
\begin{equation}\label{BA}
\underline{\beta}\langle B^{-1}v,v\rangle\leq\langle Av,v\rangle\leq \overline{\beta}\langle B^{-1}v,v\rangle,\quad\forall v\in V.
\end{equation}
The inequality \eqref{BA} is known as \emph{spectral equivalence}, or \emph{norm equivalence},
and is a common ingredient in the analysis of convergence of iterative methods for large-scale linear systems.

\subsection{Estimating  the residual}
We now show that the norm (spectral) equivalence \eqref{BA} naturally yields a two-sided estimate on $\|e\|_A$. This is  the central result in this paper.
\begin{theorem}\label{mainthm}
Let \eqref{BA} hold. Then we have the following two sided bound
\begin{equation*}
\overline{\beta}^{-1}\langle r,Sr\rangle\leq \|e\|^2_A\leq \underline{\beta}^{-1}\langle r,Sr\rangle.
\end{equation*}
\end{theorem}
\begin{proof}
Since $A$ is SPD, we use the Cauchy--Schwarz inequality to obtain
\begin{equation}\label{e:2}
\langle Ae, BAe\rangle^2\le
\langle Ae,e\rangle
\langle ABAe, BAe\rangle.
\end{equation}
The inequality~\eqref{BA} implies
\begin{equation}\label{e:3}
\langle A BAe, BAe \rangle \leq\overline{\beta}\langle B^{-1} BAe, BAe\rangle =
\overline{\beta}\langle Ae, BAe\rangle.
\end{equation}
Combining \eqref{e:2} and \eqref{e:3} yields
\begin{equation*}
\langle r, Br\rangle=\langle Ae, BAe\rangle\le
\overline{\beta}\langle Ae,e\rangle,
\end{equation*}
where we used $r=Ae$ in the first equality. The upper bound
\begin{equation*}
\langle Ae,e\rangle\leq \underline{\beta}^{-1}\langle r, Br\rangle
\end{equation*}
can be shown in a similar fashion. In summary, we have
\begin{equation}\label{BrAr}
\overline{\beta}^{-1}\langle r, Br\rangle\leq\langle Ae,e\rangle\leq \underline{\beta}^{-1}\langle r, Br\rangle.
\end{equation}
On the other hand, for any $v_h\in V_h$, \eqref{disvar} implies  \begin{align*}
    \langle Q_hr,v_h\rangle=\langle r,v_h\rangle=\langle f,v_h\rangle-\langle A_hu_h,v_h\rangle=0,
\end{align*}
i.e., $Q_hr=0$. Hence,
\begin{equation}\label{BrSr}
    Br=Sr+I_hA_h^{-1}(Q_hr)=Sr.
\end{equation}
Combining \eqref{BrAr} and \eqref{BrSr} completes the proof.
\end{proof}

Throughout the rest of this paper, $\langle r,Sr\rangle$ will serve as a (nearly) computable a posteriori error estimator that is proved to be both an upper and lower bound of the error $\|e\|_A$.
In order to derive an error estimator within our framework, the key step is to suitably select the smoother $S$ such that the spectral equivalence \eqref{BA} holds.

\subsection{Additive Schwarz smoother}
In this subsection, we construct a particular $S$ using the additive Schwarz method. For such a smoother, we present a lemma that serves as a criterion for verifying \eqref{BA}.

For $n\in\mathbb{N}$, $1\leq k\leq n$, let $V_k \subset V$ be subspaces providing a decomposition of $V$, namely,
\begin{equation}
V=\sum_{k=1}^{n} V_k.
\end{equation}
Let $I_k:V_k\hookrightarrow V$ be the natural inclusion and
$Q_k: V'\to V_k'$ denote its adjoint. We further set $A_k:=Q_k A I_k$.
Next, let $S_k: V_k^\prime\rightarrow V_k$ be spectrally equivalent to $A_k^{-1}$. More precisely, for $1\leq k\leq n$ and $v_k\in V_k$, we assume that
\begin{equation}\label{e:sk-bounds}
\underline{\gamma}\langle S^{-1}_kv_k,v_k\rangle\leq \langle A_kv_k,v_k\rangle\leq \overline{\gamma}\langle S^{-1}_kv_k,v_k\rangle,
\end{equation}
where $\underline{\gamma}, \overline{\gamma}$ are positive absolute constants. The smoother $S$ (additive Schwarz method) is then defined to be
\begin{equation*}
S:=\sum_{k=1}^n I_k S_k Q_k.
\end{equation*}
By the definition of $B$, we obtain
\begin{align*}
B=I_hA_h^{-1}Q_h+\sum_{k=1}^nI_kS_kQ_k.
\end{align*}
The norm of $B$ can be estimated using the following lemma, which can be found in e.g.,  \cite{griebel-oswald,xz-jams,widlund-book,BS2008}.
\begin{lemma}\label{l:b-inverse}
We have the following identity
\begin{eqnarray*}
\langle B^{-1} v, v\rangle = \inf_{v_h+\sum_{k=1}^n v_k=v}\langle A_hv_h,v_h\rangle+\sum_{k=1}^n\langle S_k^{-1}v_k,v_k\rangle,
\end{eqnarray*}
where the infimum is taken over $v_h\in V_h$ and $ v_k\in V_k$ for $1\leq k\leq n$.
\end{lemma}
The proof that $B$ is a good preconditioner for $A$ is standard.
We include it here for completeness and we follow the proof in~\cite{xz-jams}.
\begin{lemma}\label{cond}
For each $k$,
let
\[
  \mathcal{M}(k):=\{j: \sup_{v_j\in V_j, v_k\in V_k} a(v_j,v_k)\neq0\},
\]
and $M:=\max_{1\le k\le n}\#\mathcal{M}(k)$.
In addition, assume that for all $v\in V$, there exist $v_h\in V_h$ and  $v_k\in V_k$ with $1\leq k\leq n$ satisfying
\begin{equation}\label{stable}
    \|v_h\|_A^2+\sum_{k=1}^n\|v_k\|_A^2\leq C_{\emph{stab}}\|v\|^2_A,\quad v=v_h+\sum_{k=1}^nv_k.
\end{equation}
Then \eqref{BA} holds with constants $\overline{\beta}=2\max(1,\overline{\gamma}M)$, $\underline{\beta}=\min(1,\underline{\gamma})C^{-1}_{\emph{stab}}$.
\end{lemma}
\begin{proof}
For $v\in V$, assume the decomposition $v=v_h+\sum_{k=1}^n v_k$ with $v_h\in V_h$, $v_k\in V_k$. Direct calculation shows that
\begin{equation}\label{total}
\begin{aligned}
\|v\|_A^2&\leq2\|v_h\|_A^2+2\left\|\sum_{k=1}^nv_k\right\|_A^2\\
&=2\|v_h\|_A^2+2\sum_{j,k=1}^n\langle Av_j,v_k\rangle.
\end{aligned}
\end{equation}
The definition of $\mathcal{M}(k)$ and $M$ implies
\begin{align*}
&\sum_{j,k=1}^n\langle Av_j,v_k\rangle=
\sum_{k=1}^n
\sum_{j\in\mathcal{M}(k)}a(v_j,v_k)\\
&\quad\leq
\frac{1}{2}\sum_{k=1}^n
\sum_{j\in \mathcal{M}(k)}\|v_j\|_A^2+\|v_k\|_A^2\leq M \sum_{k=1}^n\|v_k\|_A^2.
\end{align*}
Combining the previous estimate with \eqref{total} and \eqref{e:sk-bounds} gives
\begin{equation}
\begin{aligned}
\|v\|_A^2&\leq2\langle A_hv_h,v_h\rangle+2M\sum_{k=1}^n\langle Av_k,v_k\rangle\\
&\leq2\max(1,\overline{\gamma}M)\left(\langle A_hv_h,v_h\rangle+\sum_{k=1}^n\langle S_k^{-1}v_k,v_k\rangle\right).
\end{aligned}
\end{equation}
Taking the infimum with respect to all decompositions and using Lemma \ref{l:b-inverse}, we obtain the upper bound
\begin{equation*}
    \|v\|_A^2\leq2\max(1,\overline{\gamma}M)\langle B^{-1}v,v\rangle.
\end{equation*}

For the lower bound in \eqref{BA}, let $v=v_h+\sum_{k=1}^nv_k$ be the decomposition that satisfies \eqref{stable}.
It then follows from \eqref{e:sk-bounds} and \eqref{stable} that
  \begin{align*}
  &\langle A_hv_h,v_h\rangle+\sum_{k=1}^n\langle S_k^{-1}v_k,v_k\rangle\leq \|v_h\|_A^2+\sum_{k=1}^n\underline{\gamma}^{-1}\langle A_kv_k,v_k\rangle\\
  &\quad\leq\max(1,\underline{\gamma}^{-1})\left(\|v_h\|_A^2+
  \sum_{k=1}^{n}\|v_k\|^2_{A}\right)\leq\max(1,\underline{\gamma}^{-1})C_{\text{stab}}\|v\|_A^2.
\end{align*}
Using the previous estimate and Lemma \ref{l:b-inverse}, we obtain
  \begin{align*}
  &\langle B^{-1}v,v\rangle\leq\max(1,\underline{\gamma}^{-1})C_{\text{stab}}\|v\|_A^2.
\end{align*}
The proof is complete.
\end{proof}

\section{Examples}\label{secex}
In this section, we consider the typical example of a scalar elliptic equation. Let $V=H^{1}_0(\Omega)$ where $\Omega\subset \mathbb{R}^d$ is a Lipschitz polytope.
For a given $f\in L^2(\Omega)$ and $K\in [W^{1}_{\infty}(\Omega)]^{d\times d}$, the bilinear and linear forms in equation~\eqref{variational} are:
\begin{equation*}
a(u,v):=\int_\Omega K\nabla u\cdot\nabla vdx, \quad \langle f, v\rangle:=\int_\Omega fvdx.
\end{equation*}
In addition, we assume $K$ is piecewise constant and uniformly elliptic, i.e.,
\begin{equation*}
    \underline{\alpha}|\xi|^2\leq \xi^TK(x)\xi\leq\overline{\alpha}|\xi|^2,\quad\forall\xi\in\mathbb{R}^n, x\in\Omega.
\end{equation*}
Hence, \eqref{bdd} and \eqref{coercive} holds.

Let $\mathcal{T}_h$ be a conforming and shape-regular simplicial partition of $\Omega$ aligned with discontinuities of $K$. Let $\mathcal{P}_p(D)$ denote the set of polynomials of degree at most $p$ on a domain $D$.
The subspace $V_h\subset V$ is
$$
V_h=:\{v_h\in V: v_h|_T\in\mathcal{P}_p(T)\text{ for all }T\in\mathcal{T}_h\},
$$
where $p\geq1$ is an integer.

Let $\{x_k\}_{k=1}^n$ denote the set of vertices in $\mathcal{T}_h$. For each $x_k$, let $\phi_k$ denote the continuous piecewise linear function that takes the value $1$ at $x_k$ and $0$ at other vertices. Furthermore, we denote ${\Omega}_k:= \operatorname{supp}\phi_k$ for $1\leq k\leq n$.
Obviously we have
\begin{align}
&\Omega=\bigcup_{k=1}^n \Omega_k,\quad \sum_{k=1}^n\phi_k(x)=1, \label{e:phi-1}\\
&\|\nabla \phi_k\|_{L^\infty(\Omega)}\eqsim h_k^{-1}:=(\text{diam}\Omega_k)^{-1}.\label{gradphi}
\end{align}

\subsection{A posteriori error estimates for Lagrange elements}
Now, let $V_k=H_0^1(\Omega_k)$ which is a subspace of $V=H_0^1(\Omega)$ by zero extension. The partition of unity \eqref{e:phi-1} implies
\begin{equation*}
V=\sum_{k=1}^{n} V_k.
\end{equation*}
We note that the framework also works for other local patches, as long as their union covers $\Omega$.

For a fixed $k$, the set $\mathcal{M}(k)$ defined in Lemma \ref{cond} translates into
$$
\mathcal{M}(k)=\{j: \Omega_k\cap\Omega_j \neq\emptyset\}.
$$
In this case, $M=\max_{1\le k\le n}\#\mathcal{M}(k)$ is an absolute constant by the shape-regularity of $\mathcal{T}_h$.

Throughout the rest of this paper, we adopt the notation $C_1\lesssim C_2$ provided $C_1\leq C_3C_2$ with $C_3$ being a generic constant dependent only on $K$ and $M$. We say $C_1\eqsim C_2$ provided $C_1\lesssim C_2$ and $C_2\lesssim C_1$. Given an element $T$ and a face $e$, let $h_T$ and $h_e$ denote the diameter of $T$ and $e$, respectively. The shape-regularity of $\mathcal{T}_h$ implies that
\(h_k\eqsim h_T\eqsim h_e\) if $x_k\in\overline{T}\cap\overline{e}$ and  we will use these notions interchangeably.

We set $S_k:=A_k^{-1}$ and thus $\overline{\gamma}=\underline{\gamma}=1$ in \eqref{e:sk-bounds}. The corresponding smoother $S$ yields an error estimator.
In order to show the reliability and efficiency, we need to verify \eqref{stable} in Lemma \ref{cond}.
\begin{corollary}\label{maincor}
We have the following estimate
\begin{equation*}
\|e\|^2_A\eqsim\sum_{k=1}^n\langle Q_kr,A_k^{-1}Q_kr\rangle.
\end{equation*}
\end{corollary}
\begin{proof}
To verify \eqref{stable}, we take $v_h=\Pi_h v\in V_h$, where $\Pi_h$ is a $H^1$-stable interpolation which also enjoys standard approximation properties:
\begin{equation}\label{approx}
    |\Pi_hv|^2_{H^1(\Omega)}+\sum_{k=1}^nh_k^{-2}\|v-\Pi_h v\|^2_{L^2(\Omega_k)}+|v-\Pi_h v|^2_{H^1(\Omega_k)}\lesssim |v|^2_{H^1(\Omega)}.
\end{equation}
A simple choice for $\Pi_h$ is the Cl\'{e}ment interpolation ~\cite{clement-aa}.
We now set $v_k=\phi_k (v-\Pi_hv)$. Hence,
$v=v_h+\sum_{k=1}^nv_k$ is a decomposition. It follows from \eqref{gradphi} and \eqref{approx} that
  \begin{align*}
  &\|v_h\|_A^2+\sum_{k=1}^n\|v_k\|_A^2\eqsim|\Pi_hv|_{H^1(\Omega)}^2+\sum_{k=1}^n|\phi_k(v-\Pi_hv)|_{H^1(\Omega_k)}^2\\
  &\quad\lesssim|\Pi_hv|_{H^1(\Omega)}^2+\sum_{k=1}^nh_k^{-2}\|v-\Pi_hv\|_{L^2(\Omega_k)}^2+|v-\Pi_hv|_{H^1(\Omega_k)}^2\\
  &\quad\lesssim |v|_{H^1(\Omega)}^2\lesssim\|v\|_A^2.
\end{align*}
Hence, \eqref{stable} are verified.
Finally, we conclude Corollary \ref{maincor} from Theorem \ref{mainthm} and Lemma \ref{cond}.
\end{proof}
For $\varphi\in V_k,$ we have
$$
\langle Q_kr,\varphi\rangle=\int_{\Omega_k}f\varphi dx-a(u_h,\varphi).
$$
Hence, computing $\eta_k:=A_k^{-1}Q_kr\in V_k$ amounts to solving the variational problem:
\begin{equation}\label{ctslocal}
    a(\eta_k,\varphi)=\int_{\Omega_k}f\varphi dx-a(u_h,\varphi),\quad\forall\varphi\in V_k.
\end{equation}
Taking $\varphi=\eta_k$ in \eqref{ctslocal} implies that
\begin{equation*}
\|\eta_k\|^2_A=\langle Q_kr,A_k^{-1}Q_kr\rangle.
\end{equation*}
It then follows from the previous identity  and Corollary  \ref{maincor} that
\begin{equation}\label{etak}
\|e\|^2_A\eqsim\sum_{k=1}^n\|\eta_k\|^2_A.
\end{equation}

\subsection{Computable error estimator}
Unfortunately, $\|\eta_k\|_A$ is not available in practice because \eqref{ctslocal} is local but still not fully computable.
To implement the estimator in Corollary \ref{maincor}, we consider the approximate problem: Find $\widetilde{\eta}_k\in\widetilde{V}_k$ such that
\begin{equation}\label{dislocal}
    a(\widetilde{\eta}_k,\varphi)=\int_{\Omega_k}f\varphi dx-a(u_h,\varphi),\quad\forall\varphi\in\widetilde{V}_k,
\end{equation}
where $\widetilde{V}_k\subset V_k$ is a subspace of piecewise polynomials. Roughly speaking, the approximate estimator $\left(\sum_{k=1}^n\|\widetilde{\eta}_k\|^2_A\right)^\frac{1}{2}$ is expected to be an accurate upper and lower bound of $\|e\|_A$ provided the degree of piecewise polynomials in $\widetilde{V}_k$ is sufficiently high.

Taking $\varphi=\widetilde{\eta}_k$ in \eqref{dislocal} and \eqref{ctslocal}, we obtain
\begin{equation*}
    \|\widetilde{\eta}_k\|^2_A=a(\eta_k,\widetilde{\eta}_k)\leq\|\eta_k\|_A\|\widetilde{\eta}_k\|_A.
\end{equation*}
Combining the previous estimate with \eqref{etak} provides the following computable lower bound for the error
\begin{equation}\label{etatildek}
    \sum_{k=1}^n\|\widetilde{\eta}_k\|^2_A\leq\sum_{k=1}^n\|\eta_k\|^2_A\lesssim\|e\|_A^2.
\end{equation}

To derive a computable upper bound for $\|e\|_A$, let us first write the action of the residual on $V_k=H_0^1(\Omega_k)$. We denote the set of all $(d-1)$-dimensional faces in the triangulation $\mathcal{T}_h$ by $\mathcal{E}_h$. Clearly, $\mathcal{E}_h=\mathcal{E}_h^{o}\cup\mathcal{E}_h^{\partial}$ where $\mathcal{E}_h^\partial$ denotes the set of all boundary faces and $\mathcal{E}^o$ the interior faces. We further denote $\mathcal{T}_h|_{\Omega_k}$ by $\mathcal{T}_k$ and   $\mathcal{E}_h|_{\mathring{\Omega}_k}$ by $\mathcal{E}_k$, respectively. Note that $\mathcal{E}_k$ does not include the faces on $\partial\Omega_k$. For each $T\in\mathcal{T}_h$, let
\[
r_T:=(f+\operatorname{div}K\nabla u_h)|_T.
\]
Further, for each $e\in\mathcal{E}^o_h$, let $T_1, T_2\in\mathcal{T}_h$ be the two elements sharing $e$, $n_1$ (resp.~$n_2$) the outward unit normal to $\partial K_1$ (resp.~$\partial K_2$), and
\[
r_e:=K\nabla u_h|_{T_1}\cdot n_1+K\nabla u_h|_{T_2}\cdot n_2.
\]
It then follows from \eqref{ctslocal} and integration by parts that
\begin{equation}\label{rhsrep}
    a(\eta_k,\varphi)=\sum_{T\in\mathcal{T}_k}\int_Tr_T\varphi dx+\sum_{e\in\mathcal{E}_k}\int_er_e\varphi ds,\quad\forall\varphi\in V_k.
\end{equation}
Now, let us introduce the computable quantity
$$
\zeta_k:=\left(\sum_{T\in\mathcal{T}_k}h_T^2\|r_T\|^2_{L^2(T)}+\sum_{e\in\mathcal{E}_k}h_e\|r_e\|^2_{L^2(e)}\right)^\frac{1}{2},
$$
which is the standard explicit residual error estimator.  We take
$\varphi=\eta_k$ and use~\eqref{rhsrep} and the Cauchy--Schwarz
inequality to obtain that
\begin{equation*}
\begin{aligned}
\|\eta_k\|_A^2&=a(\eta_k,\eta_k)=\sum_{T\in\mathcal{T}_k}\int_Tr_T\eta_k dx+\sum_{e\in\mathcal{E}_k}\int_er_e\eta_k ds\\
&\leq\sum_{T\in\mathcal{T}_k}\|\eta_k\|_{L^2(T)}\|r_T\|_{L^2(T)}+\sum_{e\in\mathcal{E}_k}\|\eta_k\|_{L^2(e)}\|r_e\|_{L^2(e)}\\
&\leq\zeta_k\left(\sum_{T\in\mathcal{T}_k}h_T^{-2}\|\eta_k\|^2_{L^2(T)}+\sum_{e\in\mathcal{E}_k}h^{-1}_e\|\eta_k\|^2_{L^2(e)}\right)^\frac{1}{2}.
\end{aligned}
\end{equation*}
Finally, combining the previous inequality with the trace inequality and the Poincar\'{e} inequality
$\|\eta_k\|_{L^2(\Omega_k)}\lesssim h_k\|\nabla\eta_k\|_{L^2(\Omega_k)}$ yields
\begin{equation*}
\begin{aligned}
\|\eta_k\|_A^2\lesssim\zeta_k\left(h_k^{-2}\|\eta_k\|^2_{L^2(\Omega_k)}+\|\nabla\eta_k\|^2_{L^2(\Omega_k)}\right)^\frac{1}{2}\lesssim\zeta_k\|\eta_k\|_A.
\end{aligned}
\end{equation*}
Hence, using \eqref{etak} and the previous inequality, we obtain the following computable upper bound for the error
\begin{equation}\label{zetak}
    \|e\|_A^2\eqsim\sum_{k=1}^n\|\eta_k\|_A^2\lesssim\sum_{k=1}^n{\zeta}^2_k.
\end{equation}

So far the finite element error $\|e\|_A$ is estimated
  from below and above by two different estimators. To show that
  either
  $\left(\sum_{k=1}^n\|\widetilde{\eta}_k\|_A^2\right)^\frac{1}{2}$ or
  $\left(\sum_{k=1}^n{\zeta}^2_k\right)^\frac{1}{2}$ is a
  \emph{two-sided} bound of $\|e\|_A$, we use the bubble function
  technique due Verf\"urth, which seems to be indispensible tool in
  deriving such estimates. To keep the presentation self-contained as
  much as possiblle, we give the details of deriving the two-sided
  estimates below and we note that such arguments are standard, see,
  e.g., \cite{Verfurth1996}. For each $T\in\mathcal{T}_h$ and
  $e\in\mathcal{E}_h$, the volume and face bubble functions are
  defined as
\begin{align*}
    \phi_T:=\prod_{x_k\in\overline{T}}\phi_k,\quad\phi_e:=\prod_{x_k\in\overline{e}}\phi_k,
\end{align*}
respectively. Let $\Omega_e$ denote the union of elements sharing $e$ as a face. We note that $\|\phi_T\|_{L^\infty(T)}\eqsim1,$ $\|\phi_e\|_{L^\infty(\Omega_e)}\eqsim1,$ and $\text{supp}\phi_T\subseteq T,$ $\text{supp}\phi_e\subseteq\Omega_e$. Given an integer $m\geq 0$,
it is well-known (see~\cite{Verfurth1996}) that for $v\in\mathcal{P}_m(T)$ and $w\in\mathcal{P}_m(e)$, we have the following estimates:
\begin{subequations}
\begin{align}
    &\|\phi_Tv\|_T\lesssim\|v\|_T\lesssim\|\phi_T^\frac{1}{2}v\|_T,\label{phiT}\\
    &\|\phi_ew\|_e\lesssim\|w\|_e\lesssim\|\phi_e^\frac{1}{2}w\|_e,\label{phie1}\\
    &\|E_ew\|_{\Omega_e}\eqsim  h_e^\frac{1}{2}\|w\|_e,\label{phie2}
\end{align}
\end{subequations}
where $E_ew\in\mathcal{P}_{m}(\Omega_e)$ is an extension of $w$, such that  $(E_ew)|_e=w$.

To show that $\sum_{k=1}^n\|\widetilde{\eta}_k\|_A^2$ is an upper bound for $\|e\|_A^2$, we take $\widetilde{V}_k$ in \eqref{dislocal} as
\begin{align*}
    \widetilde{V}_k:=\sum_{T\in\mathcal{T}_k}\phi_T\mathcal{P}_{p-1}(\Omega_k)+\sum_{e\in\mathcal{E}_k}\phi_e\mathcal{P}_{p-1}(\Omega_k),
\end{align*}
which is clearly a subspace of $V_k$.
Similarly to \eqref{rhsrep}, one can rewrite \eqref{dislocal} as
\begin{equation}\label{rhsrep2}
    a(\widetilde{\eta}_k,\varphi)=\sum_{T\in\mathcal{T}_k}\int_Tr_T\varphi dx+\sum_{e\in\mathcal{E}_k}\int_er_e\varphi ds,\quad\forall\varphi\in\widetilde{V}_k.
\end{equation}
Let $Q_T$ denote the $L^2$-projection onto $\mathcal{P}_{p-1}(T)$.
Using \eqref{phiT}, \eqref{rhsrep2} with $\varphi=\phi_TQ_Tr_T\in\widetilde{V}_k$, the Cauchy--Schwarz and inverse inequalities, we have
\begin{equation}\label{intervol}
\begin{aligned}
\|Q_Tr_T\|^2_T&\lesssim\int_Tr_T\varphi dx+\int_T(Q_Tr_T-r_T)\varphi dx\\
&=a(\widetilde{\eta}_k,\varphi)+\int_T(Q_Tr_T-r_T)\varphi dx\\
&\lesssim  h_T^{-1}\|\widetilde{\eta}_k\|_A\|\varphi\|_T+\|Q_Tr_T-r_T\|_T\|\varphi\|_T\\
&\lesssim\big(  h_T^{-1}\|\widetilde{\eta}_k\|_A+\|Q_Tr_T-r_T\|_T\big)\|Q_Tr_T\|_T.
\end{aligned}
\end{equation}
It then follows from \eqref{intervol},   \eqref{phiT}, and $(\text{id}-Q_T)(\text{div}K\nabla u_h)=0$ that
\begin{equation}\label{vol}
\begin{aligned}
\|r_T\|_T&\leq\|Q_Tr_T\|_T+\|r_T-Q_Tr_T\|_T\\
    &\lesssim h_T^{-1}\|\widetilde{\eta}_k\|_A+\|f-Q_Tf\|_T.
\end{aligned}
\end{equation}
On the other hand, taking $\varphi=\phi_eE_er_e\in\widetilde{V}_k$ in \eqref{rhsrep2} and using  \eqref{phie1}, \eqref{phie2}, we have
\begin{equation}\label{edge}
\begin{aligned}
\|r_e\|^2_e&\lesssim\int_er_e\varphi ds=a(\widetilde{\eta}_k,\varphi)-\sum_{T\in\mathcal{T}_k, T\subset\Omega_e}\int_Tr_T\varphi dx\\
&\lesssim  h_e^{-1}\|\widetilde{\eta}_k\|_A\|\varphi\|_{\Omega_e}+\sum_{T\in\mathcal{T}_k,T\subset\Omega_e}\|r_T\|_T\|\varphi\|_T\\
&\lesssim  \big(h_e^{-\frac{1}{2}}\|\widetilde{\eta}_k\|_A+\sum_{T\in\mathcal{T}_k,T\subset\Omega_e}h_T^\frac{1}{2}\|r_T\|_T\big)\|r_e\|_e.
\end{aligned}
\end{equation}
Hence, combining \eqref{zetak}, \eqref{vol}, \eqref{edge} and using the shape regularity of $\mathcal{T}_h$, we obtain the computable upper bound based on $\widetilde{\eta}_k$:
\begin{equation*}
    \|e\|_A^2\lesssim\sum_{k=1}^n\|\zeta_k\|^2_A\lesssim\sum_{k=1}^n\|\widetilde{\eta}_k\|^2_A+\text{osc}_{\mathcal{T}_h}(f)^2,
\end{equation*}
where
$\text{osc}_{\mathcal{T}_h}(f):=\big(\sum_{T\in\mathcal{T}_h}h_T^2\|f-Q_Tf\|_T^2\big)^\frac{1}{2}$ is called the data oscillation in the literature. Compared with $\|e\|_A,$ the quantity $\text{osc}_{\mathcal{T}_h}(f)$ is a higher order term provided $f$ is piecewise smooth.

Similarly, using \eqref{rhsrep} with $\varphi=\phi_TQ_Tr_T$, $\varphi=\phi_eE_er_e$, and \eqref{etak} we obtain
\begin{align*}
    \sum_{k=1}^n\|\zeta_k\|^2_A\lesssim\sum_{k=1}^n\|\eta_k\|^2_A+\text{osc}_{\mathcal{T}_h}(f)^2\lesssim\|e\|_A^2+\text{osc}_{\mathcal{T}_h}(f)^2,
\end{align*}
which is a lower bound based on $\zeta_k$.


\section{Concluding remarks}\label{secconclusion}
For SPD problems, we have shown how preconditioning can be used to derive a posteriori error estimates. Extensions of this abstract theoretical framework and its application to derive estimators for indefinite, nonconforming, and discontinuous Galerkin methods are ongoing. A close inspection of the arguments shows that not only preconditioning can give a unified way to derive a posteriori error estimators. This is a two-way street: the a posteriori error estimators may provide efficient smoothers for multilevel methods. For example, the operator $S$ we have introduced in our framework is a clear analogue of smoothing (relaxation) operator. We hope that some of the error indicators and estimators may give efficient smoothers in case of non-symmetric and or indefinite problems which are, in general, hard to precondition.

\end{document}